\documentclass[a4paper,12pt,reqno]{amsart}
\usepackage{url}
\usepackage{a4wide}
\usepackage{amsmath,amsfonts,amssymb, latexsym, amsbsy,amsthm,epsfig,color,mathrsfs,nicefrac}
\usepackage{hyperref}

\begin{document}
 \bibliographystyle{plain}

 \newtheorem{theorem}{Theorem}
 \newtheorem{lemma}{Lemma}
 \newtheorem{corollary}{Corollary}
 \newtheorem{problem}{Problem}
 \newtheorem{conjecture}{Conjecture}
 \newtheorem{definition}{Definition}
 \newcommand{\mc}{\mathcal}
 \newcommand{\rar}{\rightarrow}
 \newcommand{\Rar}{\Rightarrow}
 \newcommand{\lar}{\leftarrow}
 \newcommand{\lrar}{\leftrightarrow}
 \newcommand{\Lrar}{\Leftrightarrow}
 \newcommand{\zpz}{\mathbb{Z}/p\mathbb{Z}}
 \newcommand{\mbb}{\mathbb}
 \newcommand{\A}{\mc{A}}
 \newcommand{\B}{\mc{B}}
 \newcommand{\cc}{\mc{C}}
 \newcommand{\D}{\mc{D}}
 \newcommand{\E}{\mc{E}}
 \newcommand{\F}{\mc{F}}
 \newcommand{\G}{\mc{G}}
  \newcommand{\ZG}{\Z (G)}
 \newcommand{\FN}{\F_n}
 \newcommand{\I}{\mc{I}}
 \newcommand{\J}{\mc{J}}
 \newcommand{\M}{\mc{M}}
 \newcommand{\nn}{\mc{N}}
 \newcommand{\qq}{\mc{Q}}
 \newcommand{\U}{\mc{U}}
 \newcommand{\X}{\mc{X}}
 \newcommand{\Y}{\mc{Y}}
 \newcommand{\itQ}{\mc{Q}}
 \newcommand{\C}{\mathbb{C}}
 \newcommand{\R}{\mathbb{R}}
 \newcommand{\N}{\mathbb{N}}
 \newcommand{\Q}{\mathbb{Q}}
 \newcommand{\Z}{\mathbb{Z}}
 \newcommand{\ff}{\mathfrak F}
 \newcommand{\fb}{f_{\beta}}
 \newcommand{\fg}{f_{\gamma}}
 \newcommand{\gb}{g_{\beta}}
 \newcommand{\vphi}{\varphi}
 \newcommand{\whXq}{\widehat{X}_q(0)}
 \newcommand{\Xnn}{g_{n,N}}
 \newcommand{\lf}{\left\lfloor}
 \newcommand{\rf}{\right\rfloor}
 \newcommand{\lQx}{L_Q(x)}
 \newcommand{\lQQ}{\frac{\lQx}{Q}}
 \newcommand{\rQx}{R_Q(x)}
 \newcommand{\rQQ}{\frac{\rQx}{Q}}
 \newcommand{\elQ}{\ell_Q(\alpha )}
 \newcommand{\oa}{\overline{a}}
 \newcommand{\oI}{\overline{I}}
 \newcommand{\dx}{\text{\rm d}x}
 \newcommand{\dy}{\text{\rm d}y}
\newcommand{\cal}[1]{\mathcal{#1}}
\newcommand{\cH}{{\cal H}}
\newcommand{\diam}{\operatorname{diam}}

\parskip=0.5ex

\title[Quantitative ergodic theorems for weakly integrable functions]{Quantitative ergodic theorems\\for weakly integrable functions}
\author{Alan~Haynes}
\thanks{Research supported by EPSRC grant EP/J00149X/1.}
\address{School of Mathematics, University of Bristol, Bristol UK }
\email{alan.haynes@bristol.ac.uk}

\begin{abstract}
Under suitable hypotheses we establish a quantitative pointwise ergodic theorem which applies to trimmed Birkhoff sums of weakly integrable functions.
\end{abstract}

\allowdisplaybreaks

\maketitle

\section{Introduction}
Suppose that $T$ is an ergodic measure preserving transformation of a probability space $(X,\mu)$ and that $f$ is a non-negative measurable function on $X$. If $f\in \mathrm{L}^1(X,\mu)$ then the Birkhoff Ergodic Theorem \cite[Theorem 1.14]{Walters1982} tells us that
\begin{equation*}
\lim_{N\rar\infty}\frac{1}{N}\sum_{n=0}^{N-1}f(T^nx)=\int_X f~d\mu ~\text{ for a.e. }x\in X.
\end{equation*}
Several mathematicians have contemplated, in different contexts, the possibility of extending this theorem to non-integrable functions. In this more general setup the problem is to determine if there is a sequence $\{F_N\}\subseteq\R$ such that
\begin{equation}\label{eqnergavg1}
\lim_{N\rar\infty}\frac{1}{F_N}\sum_{n=0}^{N-1}f(T^nx)=1~\text{ for $\mu$-a.e. }~x\in X,
\end{equation}
Note that since $T$ is ergodic, if the limit on the left of (\ref{eqnergavg1}) exists then it must be constant a.e. (\cite[Theorem 1.6]{Walters1982}). Aaronson showed in \cite[Theorem 1]{Aaronson1977} that if $T$ is any ergodic transformation of a probability space and $f$ is a non-negative measurable function for which (\ref{eqnergavg1}) is satisfied for some choice of $\{F_N\}$, then $f$ must be integrable.

On the other hand it was shown by Diamond and Vaaler \cite{DiamondVaaler1986} that for the special case of simple continued fractions (i.e. the Gauss map on $\R/\Z$), the ergodic theorem can be extended to weakly integrable functions by removing at most one term from each of the Birkhoff sums. This result was recently generalized by Aaronson, Nakada, and Natsui \cite{AaronsonNakada2003, NakadaNatsui2002, NakadaNatsui2003}.

The goal of this paper is to present a quantitative version of these results which applies to a large class of well known dynamical systems. Our results do require estimates on the rate of mixing of the transformations involved, but such hypotheses are to be expected, as for many dynamical systems there is no hope of obtaining the types of theorems that we are presenting (cf. \cite{Avigad2009, Avigadetal2010, Kachurovskii1996}). To strengthen this point we demonstrate in Theorem \ref{thmmixingexample} that even the hypothesis of strong mixing alone is not enough to guarantee the results in our first two theorems.

Let $\A=\{A_i\}_{i\in\N}$ be a measurable partition of a probability space $(X,\mu)$ and let $T$ be a measure preserving transformation of the space. Further suppose that there is a function $g:\N\rar [0,\infty)$ with $g(N)=o(N)$ and
\begin{equation}\label{eqnuniformhyp1}
\sum_{n=0}^N\left(\frac{\mu(A_i\cap T^{-n}A_j)}{\mu (A_i)\mu (A_j)}-1\right)\le g(N)\quad\text{for all}\quad i,j,N\in\N.
\end{equation}
As a special case of one of our main results, we obtain the following.
\begin{theorem}\label{thmweakint(cor)}
Suppose $(X,\mu), T, \A,$ and $g$ are as above and that $f:X\rar [0,\infty )$ is measurable with respect to the $\sigma-$algebra generated by $\A$. Given $\epsilon >0$ let
\[F(N)=\int_{\{f\le N\log^{1/2+\epsilon}N\}}f~d\mu\quad\text{ and }\quad G(N)=\sum_{n=1}^Ng(n).\]
If $g(N)=o(N^{1/2})$ and $F(N+1)-F(N)\ll 1/N$ then there is a function $\delta_\epsilon:X\times\N\rar \{0,1\}$ such that for a.e. $x\in X$,
\begin{align*}
\sum_{n=0}^{N-1}f(T^nx)=&N F(N)+\delta_\epsilon(x,N)\cdot\max_{0\le n<N}f(T^nx)+O_\epsilon\left((N+G(N))^{2/3}\log^{5/3+\epsilon}N\right).
\end{align*}
\end{theorem}
For a large class of Markov maps this theorem gives an error term of $O(N^{2/3}\log^{5/3+\epsilon}N)$. As a prototypical application we give the following quantitative version of \cite[Corollary 1]{DiamondVaaler1986}.
\begin{corollary}\label{corcontfrac}
For any $\epsilon>0$ there is a function $\delta_\epsilon:\R\times\N\rar \{0,1\}$ such that for a.e. $x\in\R$ (with respect to Lebesgue measure),
\begin{align*}
\sum_{n=1}^Na_n(x)=&\frac{N}{\log 2}\sum_{n\le N\log^{1/2+\epsilon}N}n\log\left(1+\frac{1}{n(n+2)}\right)\\
&+\delta_\epsilon (x,N)\cdot\max_{1\le n\le N}a_n(x)+O_\epsilon(N^{2/3}\log^{5/3+\epsilon}N),
\end{align*}
where $x=[a_0(x);a_1(x),a_2(x),\ldots]$ denotes the simple continued fraction expansion of $x$.
\end{corollary}
One interesting feature of these results is that the main terms vary with $\epsilon$ in a subtle way. For example in the continued fractions result the first main term starts off as
\[\frac{N\log N}{\log 2}+\frac{(1/2+\epsilon)N\log\log N}{\log 2},\]
but the error term is of much smaller order. The explanation is that this dependence on $\epsilon$ is balanced by the choice of the cutoff function $\delta_\epsilon$.

To state our main results in full generality, for any increasing function $\phi :\R^+\rar\R^+$ we define $\mathrm{L}^{\phi, w}=\mathrm{L}^{\phi,w}(X,\mu)$ to be the collection of measurable functions $f:X\rar\R$ for which
\begin{equation*}
K_\phi(f)=\sup_{\lambda >0}\phi(\lambda)\cdot \mu\{x\in X: |f(x)|>\lambda\}<\infty.
\end{equation*}
If $\phi (\lambda)=\lambda$ then $\mathrm{L}^{\phi,w}$ is the usual space of weakly integrable functions.
\begin{theorem}\label{thmweakint}
Suppose $(X,\mu), T, \A,$ and $g$ satisfy (\ref{eqnuniformhyp1}) and suppose that $f\in \mathrm{L}^{\phi, w}$ is non-negative and measurable with respect to $\A$. Given $\epsilon >0$ set $\tau_\phi (N)=\phi^{-1}(N\log^{1/2+\epsilon} N)$, and let
\begin{equation*}
F_1(N)=\int_{\{f\le\tau_\phi(N)\}}f~d\mu,\quad F_2(N)=\int_{\{f\le\tau_\phi(N)\}}f^2~d\mu,
\end{equation*}
\[G(N)=\sum_{n=1}^Ng(n),\quad\text{ and }\quad F_3(N)=F_1(N)^2(N+G(N))+F_2(N).\]
If
\begin{align*}
&(N+1)F_1(N+1)-NF_1(N)~\ll~ F_3(N+1)-F_3(N)~\ll~ F_3(N)^{2/3}
\end{align*}
then there is a function $\delta_\epsilon:X\times\N\rar \{0,1\}$ such that for a.e. $x\in X$
\begin{align*}
\sum_{n=0}^{N-1}f(T^nx)=&N F_1(N)+\delta_\epsilon(x,N)\cdot\max_{0\le n<N}f(T^nx)+O_\epsilon\left(F_3(N)^{2/3}\log^{1/3+\epsilon} F_3(N)\right).
\end{align*}
\end{theorem}
In practice the hypotheses on $F_1$ and $G$ are easy to check. The main differences between our results and those in \cite{AaronsonNakada2003, NakadaNatsui2002, NakadaNatsui2003} are that ours are quantitative and, although our hypotheses on $g(N)$ are in general stronger than those in \cite[Theorem 1.1]{AaronsonNakada2003}, we only require the removal of at most one term from each Birkhoff sum.

Theorem \ref{thmweakint} will be derived from the following quantitative result for sums of Birkhoff sums.
\begin{theorem}\label{thmseqoffns}
Suppose that $(X,\mu),~T,$ and $\A$ are as above and that (\ref{eqnuniformhyp1}) is satisfied. Let $\{f_n\}_{n\in\N}$ be a sequence of non-negative functions in $\mathrm{L}^2$, measurable with respect to $\A$, and for each $N\in\N$ let
\begin{equation*}
F_1(N)=\int_X\sum_{n=1}^Nf_n~d\mu,\quad F_2(N)=\int_X\left(\sum_{n=1}^Nf_n\right)^2~d\mu,
\end{equation*}
\[G(N)=\sum_{n=1}^Ng(n),\quad\text{ and }\quad F_3(N)=F_1(N)^2(N+G(N))+F_2(N).\]
If
\begin{align}
&(N+1)F_1(N+1)-NF_1(N)~\ll~ F_3(N+1)-F_3(N)~\ll~ F_3(N)^{2/3}\label{eqnthm2.1}
\end{align}
then for any $\epsilon>0$ and for a.e. $x\in X$ we have that
\begin{align*}
\sum_{m=0}^{N-1}\sum_{n=1}^Nf_n(T^mx)=&NF_1(N)+O_\epsilon\left(F_3(N)^{2/3}\log^{1/3+\epsilon} F_3(N)\right).
\end{align*}
\end{theorem}
From Theorem \ref{thmseqoffns} we also obtain a quantitative version of the classical ergodic theorem.
\begin{corollary}\label{corintfnergthm}
If (\ref{eqnuniformhyp1}) is satisfied, if $f\in\mathrm{L}^2$ is measurable with respect to $\A$, and if
\begin{equation}\label{eqncor3.1}
g(N+1)\ll (N+G(N))^{2/3},
\end{equation}
then for almost every $x\in X$ we have that
\begin{equation*}
\frac{1}{N}\sum_{n=0}^{N-1}f(T^nx)=\int_Xf~d\mu+O_\epsilon\left(N^{-1}(N+G(N))^{2/3}\log^{1/3+\epsilon}(N+G(N))\right).
\end{equation*}
\end{corollary}
We conclude with a discussion of how our hypotheses on $T$ compare to the usual notions of ergodicity and mixing. On one hand it is easy to see that if $\A$ is finite then (\ref{eqnuniformhyp1}) will be satisfied whenever $T$ is ergodic. On the other hand when $\A$ is infinite then even strong mixing is not enough to ensure that (\ref{eqnuniformhyp1}) is satisfied. Furthermore, as the following theorem demonstrates, the hypothesis of strong mixing by itself is not strong enough to draw the conclusion in Theorem \ref{thmweakint}.
\begin{theorem}\label{thmmixingexample}
For any choice of $\{F_N\}\subseteq\R^+$ and $\delta:\R/\Z\times\N\rar\{0,1\}$, the quantities
\begin{equation*}
\frac{1}{F_N}\left(\sum_{n=0}^{N-1}\left\lfloor\frac{1}{\{2^nx\}}\right\rfloor-\delta(x,N)\cdot\max_{0\le n<N}\left\lfloor\frac{1}{\{2^nx\}}\right\rfloor\right)
\end{equation*}
either diverge for a.e. $x\in\R/\Z$, or they converge to $0$ a.e. (with respect to Lebesgue measure), as $N\rar\infty$.
\end{theorem}
This example shows the importance of the uniformity in hypothesis (\ref{eqnuniformhyp1}).

\emph{Acknowledgments:} The author would like to thank Jeffrey Vaaler, Tom Ward, and Manfred Einsiedler for helpful conversations concerning the results in this paper.

\section{Proofs of Theorem \ref{thmseqoffns} and Corollary \ref{corintfnergthm}}
For each $N\in\N$ let
\[\Phi_N(x)=\sum_{m=0}^{N-1}\sum_{n=1}^Nf_n(T^mx).\]
Since $T$ is a measure preserving transformation we have that
\begin{equation*}
\int_X\Phi_N~d\mu=NF(N).
\end{equation*}
For the second moments we have
\begin{align}\label{eqnproofthm2.1}
\int_X\Phi_N^2~d\mu &=\sum_{n,n'=1}^N\sum_{m,m'=0}^{N-1}\int_Xf_n(T^mx)f_{n'}(T^{m'}x)~d\mu.
\end{align}
Now for each $i,n\in\N$ let us denote the value of $f_n$ on $A_i$ by $\alpha_{n,i},$ so that
\begin{align}
&\sum_{m=0}^{N-1}\sum_{m'=m}^{N-1}\int_Xf_n(T^mx)f_{n'}(T^{m'}x)~d\mu\nonumber\\ &\qquad=\sum_{m=0}^{N-1}\sum_{m'=0}^{N-1-m}\int_Xf_n(x)f_{n'}(T^{m'}x)~d\mu\nonumber\\
&\qquad=\sum_{\substack{i,j=1\\ i\not=j}}^\infty\alpha_{n,i}\alpha_{n',j}\sum_{m=0}^{N-1}\sum_{m'=0}^{N-1-m}\mu(A_i\cap T^{-m'}A_j)\label{eqn2ndmom1}\\
&\qquad\quad +\sum_{i=1}^\infty\alpha_{n,i}\alpha_{n',i}\sum_{m=0}^{N-1}\sum_{m'=0}^{N-1-m}\mu(A_i\cap T^{-m'}A_i).\label{eqn2ndmom2}
\end{align}
Using (\ref{eqnuniformhyp1}) in (\ref{eqn2ndmom1}) gives
\begin{align}
&\sum_{\substack{i,j=1\\ i\not=j}}^\infty\alpha_{n,i}\alpha_{n',j}\sum_{m=0}^{N-1}\sum_{m'=0}^{N-1-m}\mu(A_i\cap T^{-m'}A_j)\nonumber\\
&\qquad=\sum_{\substack{i,j=1\\i\not= j}}^\infty\alpha_{n,i}\alpha_{n',j}\mu(A_i)\mu(A_j)\sum_{m=0}^{N-2}\left(N-m+O(g(N-m))\right).\label{eqn2ndmom3}
\end{align}
It is worth pointing out that here we have also used the fact that the summand corresponding to $m=N-1$ is $0$, since $i\not=j$. For (\ref{eqn2ndmom2}) we have that
\begin{align}
&\sum_{i=1}^\infty\alpha_{n,i}\alpha_{n',i}\sum_{m=0}^{N-1}\sum_{m'=0}^{N-1-m}\mu(A_i\cap T^{-m'}A_i)\nonumber\\
&\qquad=\sum_{i=1}^\infty\alpha_{n,i}\alpha_{n',i}\mu(A_i)^2\sum_{m=0}^{N-2}\left(N-m+O(g(N-m))\right)+\int_X f_nf_{n'}d\mu .\label{eqn2ndmom4}
\end{align}
Combining (\ref{eqn2ndmom3}) and (\ref{eqn2ndmom4}) gives
\begin{align*}
&\sum_{m=0}^{N-1}\sum_{m'=m}^{N-1}\int_Xf_n(T^mx)f_{n'}(T^{m'}x)~d\mu\\
&\qquad =\int_Xf_n~d\mu\cdot\int_Xf_{n'}~d\mu\cdot\left(\frac{N^2+N}{2}+O(G(N))\right)+\int_X f_nf_{n'}d\mu,
\end{align*}
and adding this to the complementary sum in (\ref{eqnproofthm2.1})  (i.e. with $m'\le m$) we have that
\begin{align*}
&\int_X\left(\Phi_N-NF_1(N)\right)^2~d\mu\\
&\qquad\le \sum_{n,n'=1}^N\left(\sum_{\substack{m,m'=0\\m\le m'}}^{N-1}+\sum_{\substack{m,m'=0\\m'\le m}}^{N-1}\right)\int_Xf_n(T^mx)f_{n'}(T^{m'}x)~d\mu-N^2F_1(N)^2\\
&\qquad\ll F_1(N)^2(N+G(N))+F_2(N)=F_3(N).
\end{align*}
Now for each $j\in\N$ let $N_j$ be the smallest integer with
\[F_3(N_j)> j^3\log^{1+\epsilon}(j+1).\]
By Chebyshev's Inequality and hypothesis (\ref{eqnthm2.1}) we have that
\begin{align*}
\sum_{j=1}^\infty\mu\left\{x\in X:|\Phi_{N_j}(x)-N_jF_1(N_j)|>j^2\log^{1+\epsilon}(j+1)\right\}\ll \sum_{j=1}^\infty\frac{1}{j\log^{1+\epsilon}(j+1)}<\infty,
\end{align*}
and the Borel-Cantelli Lemma allows us to conclude that for a.e. $x\in X$
\begin{equation}\label{eqnproofthm2.2}
\Phi_{N_j}(x)=N_jF_1(N_j)+O_\epsilon\left(j^2\log^{1+\epsilon}(j+1)\right)\quad\text{as}\quad j\rar\infty.
\end{equation}
Since the functions $f_n$ are non-negative we have that $\Phi_{N_j}\le\Phi_N\le\Phi_{N_{j+1}}$ whenever $N_j\le N\le N_{j+1}$. Using (\ref{eqnthm2.1}) we see that
\begin{align*}
&N_{j+1}F_1(N_{j+1})-N_jF_1(N_j)\\
&\qquad\ll F_3(N_{j+1})-F_3(N_j)\\
&\qquad=(j+1)^3\log^{1+\epsilon}(j+2)-j^3\log^{1+\epsilon}(j+1)+O(j^2\log^{2/3(1+\epsilon)}(j+1))\\
&\qquad\ll j^2\log^{1+\epsilon}(j+1).
\end{align*}
Thus we can interpolate between $N_j$ and $N_{j+1}$ in (\ref{eqnproofthm2.2}) to obtain
\begin{equation*}
\Phi_{N}(x)=NF_1(N)+O_\epsilon\left(j^2\log^{1+\epsilon}(j+1)\right),~N_{j-1}\le N<N_j,
\end{equation*}
for almost every $x\in X$. Finally for $N\ge N_{j-1}$ we have that
\begin{equation*}
j^2\log^{1+\epsilon}(j+1)\ll F_3(N)^{2/3}\log^{1/3+\epsilon}F_3(N),
\end{equation*}
which finishes the proof.

Corollary \ref{corintfnergthm} follows directly from Theorem \ref{thmseqoffns} by setting $f_1=f$ and $f_n=0$ for $n\ge 2.$ In this case $F_1(N)$ is constant and thus (\ref{eqnthm2.1}) is satisfied if and only if (\ref{eqncor3.1}) is.

\section{Proofs of Theorems \ref{thmweakint(cor)} and \ref{thmweakint} and Corollary \ref{corcontfrac}}
The proof of Theorem \ref{thmweakint} requires one more important piece of information.
\begin{lemma}\label{lemonelargevalue}
Assume that the hypotheses of Theorem \ref{thmweakint} are satisfied and that $f\notin \mathrm{L}^1$. Then for a.e. $x\in X$ there are at most finitely many $N\in\N$ for which
\begin{equation*}
\#\{0\le n<N~:~f(T^nx)>\tau_\phi(N)\}>1.
\end{equation*}
\end{lemma}
\begin{proof}
For each $N\in\N$ let
\begin{equation*}
B_N=\sum_{\substack{m,n=0\\m\not= n}}^{2N-1}\mu\left\{x\in X:f(T^mx)>\tau_\phi(N),f(T^nx)>\tau_\phi(N)\right\}.
\end{equation*}
Also if $N\in\N$ then let $\{i_j\}_{j\in\N}$ be the sequence of distinct integers for which
\begin{equation*}
\{x\in X:f(x)>\tau_\phi(N)\}=\bigcup_{j=1}^\infty A_{i_j}.
\end{equation*}
By our hypotheses on $T$ we have that
\begin{align*}
B_N&=2\sum_{m=0}^{2N-2}\sum_{n=m+1}^{2N-1}\mu\left\{f(x)>\tau_\phi(N),f(T^{n-m}x)>\tau_\phi(N)\right\}\\
&=2\sum_{j,k=1}^\infty\sum_{m=0}^{2N-2}\sum_{n=1}^{2N-1-m}\mu\{A_{i_j}\cap T^{-n}A_{i_k}\}\\
&=2\sum_{j,k=1}^\infty\mu(A_{i_j})\mu(A_{i_k})\sum_{m=0}^{2N-2}(2N-m-1+O(g(2N-m)))\\
&\ll\mu\{f(x)>\tau_\phi(N)\}^2N^2.
\end{align*}
Since $f\in \mathrm{L}^{\phi,w}$ we deduce that
\[B_N\ll (\log N)^{-(1+\epsilon)}.\]
Now since
\begin{equation*}
\sum_{k=1}^\infty B_{2^k}<\infty,
\end{equation*}
the Borel-Cantelli Lemma guarantees that for almost every $x\in X$ there are only finitely many $k\in\N$ for which $f(T^mx)>\tau_\phi(2^k)$ and $f(T^nx)>\tau_\phi(2^k)$ for two distinct integers $0\le m<n<2^{k+1}$. Now the proof is finished with the observation that if $2^k\le N <2^{k+1}$ then $\tau_\phi(N)\ge \tau_\phi(2^k)$.
\end{proof}
Note that the proof of Lemma \ref{lemonelargevalue} used both (\ref{eqnuniformhyp1}) and the hypothesis that $f\in \mathrm{L}^{\phi,w}$. It appears that neither of these conditions can be weakened. For the proof of Theorem \ref{thmweakint} we now apply Theorem \ref{thmseqoffns} with
\[f_n(x)=\begin{cases}f(x)&\text{if}\quad\tau_\phi(n-1)< f(x)\le\tau_\phi(n) ,\\0&\text{otherwise.}\end{cases}\]
Then the functions $F_1(N)$ in the statements of both theorems are the same. If $f\in\mathrm{L}^1$ then Theorem \ref{thmweakint} follows directly from Theorem \ref{thmseqoffns}. If $f\notin\mathrm{L}^1$ then Theorem \ref{thmweakint} is a consequence of Theorem \ref{thmseqoffns} together with Lemma \ref{lemonelargevalue}.

For the proof of Theorem \ref{thmweakint(cor)} first note that under the hypotheses there,
\begin{align*}
&N(\log^{1/2+\epsilon}N)\cdot\mu\{N\log^{1/2+\epsilon} N<f\le (N+1)\log^{1/2+\epsilon}(N+1)\}\\
&\qquad \le F(N+1)-F(N)\\
&\qquad \ll 1/N.
\end{align*}
It follows easily from this that $f\in\mathrm{L}^{\lambda,w}$ and that
\begin{align*}
\int_{\{f\le N\log^{1/2+\epsilon}N\}}f~d\mu &\ll \log N\\\intertext{ and}
\int_{\{f\le N\log^{1/2+\epsilon}N\}}f^2~d\mu &\ll N\log^{1/2+\epsilon} N.
\end{align*}
These observations together with the assumption that $g(N)=o(N^{1/2})$ easily imply that both inequalities in (\ref{eqnthm2.1}) are satisfied. Theorem \ref{thmweakint(cor)} then follows directly from Theorem \ref{thmweakint}.

Corollary \ref{corcontfrac} follows Theorem \ref{thmweakint(cor)} by considering the Birkhoff sums of the function $f_g:\R/\Z\rar\R,~f_g(x)=a_1(x)$ with respect to the Gauss map $T_g:\R/\Z\rar\R/\Z,~T_g(x)=[0;a_2(x),a_3(x),\ldots ]$, which is ergodic with respect to the absolutely continuous Borel measure $\mu_g$ defined by
\[\mu_g(E)=\frac{1}{\log 2}\int_E\frac{dx}{1+x}.\]
In this case it is well known that $g(N)\ll 1$, and the result follows immediately.

\section{Proof of Theorem \ref{thmmixingexample}}
Let $X=\R/\Z$ equipped with Lebesgue measure, and let $T:X\rar X$ be the map $T(x)=2x.$ This is a strong mixing measure preserving transformation of $\R/\Z$ (\cite[Theorems 1.10, 1.28]{Walters1982}). Let $f=f_g$ be the function defined at the end of the previous section and note that this function is in weak-$\mathrm{L}^1$ but not in $\mathrm{L}^1$. Therefore by Aaronson's theorem \cite[Theorem 1]{Aaronson1977}, for any $\{F_N\}\subseteq\R^+$ we have that
\begin{align}\label{eqnproofthm3.1}
\liminf_{N\rar\infty}\frac{1}{F_N}\sum_{n=0}^{N-1}f(T^nx)=0~\text{ a.e. or }~\limsup_{N\rar\infty}\frac{1}{F_N}\sum_{n=0}^{N-1}f(T^nx)=\infty~\text{ a.e. (or both).}
\end{align}
Now suppose that $\{F_N\}\subseteq\R$ and $\delta:\R/\Z\times\N\rar\{0,1\}$ are chosen so that
\begin{equation}\label{eqnproofthm3.2}
\frac{1}{F_N}\left(\sum_{n=0}^{N-1}f(T^nx)-\delta(x,N)\cdot\max_{0\le n<N}f(T^nx)\right)
\end{equation}
converges a.e. to $f^*:\R/\Z\rar\R$, as $N\rar\infty$. Then $f^*$ is measurable and $T$ invariant and it follows from \cite[Theorem 1.6]{Walters1982} that $f^*(x)=c$ a.e. for some $c\in\R$. Suppose that $c\not= 0$, and let $B\subseteq\R/\Z$ be the set of $x$ for which (\ref{eqnproofthm3.1}) holds but (\ref{eqnproofthm3.2}) converges to $c$. Then $B$ is a set of measure $1$ and since $c\not= 0$ it must be the second of the two options from (\ref{eqnproofthm3.1}) which holds for all $x\in B$.
Choose $x\in B$ and let $1\le N_1<N_2<\cdots$ be the sequence of positive integers $N$ for which $\delta(x,N)$ takes the value $1$. Then we must have that
\begin{equation*}
\limsup_{j\rar\infty}\frac{\max_{0\le n<N_j}f(T^nx)}{F_{N_j}}=\infty.
\end{equation*}
Since $F_N\gg N$ we have that
\[\max_{0\le n<N_j}f(T^nx)>2\]
for all sufficiently large $j$. Thus if we let $0\le n_j<N_j$ denote the integer for which
\[f(T^{n_j}x)=\max_{0\le n<N_j}f(T^nx)\]
we have that
\[f(T^{n_j+1}x)=(1/2)\max_{0\le n<N_j}f(T^nx).\]
Finally it is not difficult to show that $n_j<N_j-1$ for infinitely many $j$ (otherwise we would have to have that $c=0$). This means that for infinitely many $j\in\N$,
\[\frac{1}{F_{N_j}}\left(\sum_{n=0}^{N_j-1}f(T^nx)-\delta(x,N_j)\cdot\max_{0\le n<N_j}f(T^nx)\right)\ge\frac{\max_{0\le n<N_j}f(T^nx)}{2F_{N_j}}.\]
However the right hand side tends to infinity as $j\rar\infty$, and this is a contradiction. Therefore if (\ref{eqnproofthm3.2}) converges a.e. then it must converge to $0$.

\end{document}